\documentclass[10pt,reqno,a4paper]{amsart}%
\usepackage{amsfonts}
\usepackage[pdftex]{graphicx}
\usepackage{caption}
\usepackage{amsmath}
\usepackage{amssymb,amsthm,hyperref,caption}
\usepackage{amssymb}%
\providecommand{\U}[1]{\protect\rule{.1in}{.1in}}
\newtheorem{theorem}{Theorem}
\newtheorem{lem}[theorem]{Lemma}

\newtheorem{prop}{Proposition}

\theoremstyle{definition}

\newtheorem{remark}{Remark}
\newtheorem{example}{Example}

\DeclareMathOperator{\Li}{Li}
\DeclareMathOperator{\Erfc}{Erfc} 

\newcommand{\E}{\ensuremath{\mathbb{E}}} 
\newcommand{\ap}{\ensuremath{\mathcal{A}}}

\email{kuba@technikum-wien.at}
\email{mordechailevy@gmail.com}

\keywords{Multiple zeta values, Nielsen's generalized polylogarithms, Appell sequence, Asymptotic expansion}
\subjclass[2000]{11M32, 33B30, 60C05}
\begin{document}
\author[M.~Kuba]{Markus Kuba}
\address{Markus Kuba\\
Department Applied Mathematics and Physics\\
University of Applied Sciences - Technikum Wien\\
H\"ochst\"adtplatz 5, 1200 Wien}
\author[M.~Levy]{Moti Levy}
\address{Moti Levy\\
The Open University of Israel, Department of Mathematics and Computer Science}
\title[Asymptotic expansions of integrals and Nielsen's polylogarithms]{Asymptotic expansions of integrals and Nielsen's polylogarithms}

\begin{abstract}
This article derives full asymptotic expansions for
integrals of the form
\[
\int_{0}^{1}f(u)(1+q\cdot u^{n})^{w/n}du
\]
as $n\rightarrow\infty$, with parameters real $w\neq 0$ and $q\in(-1,1]$, or positive $w$ for $q=-1$.
We relate the coefficients of the asymptotic expansions to Nielsen's
generalized polylogarithms. For $q=-1$, we obtain an expansion in terms of
multiple zeta values, which in this setting, reduce to ordinary zeta values. A
key point is that for $q=1$, the integrals typically produce alternating
multiple zeta values; we formulate a precise symmetry constraint on the
relevant coefficient sequence under which all coefficients reduce to
polynomials in ordinary zeta values. We also translate this symmetry into a
statement about a binomial transform, and we verify the condition for several
classical Appell-type families, like Euler, Bernoulli, Genocchi, and Hermite.
Finally, we obtain precise results about the convergence of norms of random variables.

\end{abstract}
\maketitle

\section{Introduction}

Given a real function $f$ and parameters $q\in(-1,1]$, $w\neq0$ or $q=-1$ and positive real $w$. Let
$I_{n}=I_{n}(q,w)$ be defined as
\begin{equation}
I_{n}=\int_{0}^{1}f(u)(1+q\cdot u^{n})^{w/n}du.\label{eq:int1}%
\end{equation}
We assume that $f$ is analytic on the unit interval and that it satisfies
\[
\int_{0}^{1}|f(u)|du<\infty,\quad\int_{0}^{1}|f(u)|(1+q\cdot u^{n}%
)^{w/n}\,du<\infty.
\]
The evaluation and asymptotic expansion of the family of integrals $I_{n}$ for
integer $n$ tending to infinity is closely related to many questions. First,
we point out several problems posed in the American Mathematical Monthly. In
particular, we single out the evaluation and asymptotic expansions of the
integral
\begin{equation}
\int_{0}^{\frac{\pi}{2}}\bigl|\sin^{n}(x)-\cos^{n}(x)\bigr|^{\frac{1}{n}%
}\,dx,\label{eq:int2}%
\end{equation}
as recently proposed~\cite{AMM2025}, asking for the evaluation of the first
two coefficients of the asymptotic expansion of the integral. We will show in
Section~\ref{sec:Apps} that~\eqref{eq:int2} is covered by our general form
with $q=-1$, $w=1$ and $f(u)=2(1+u^{2})^{-3/2}$.
Second, we turn to the theory of multiple zeta values. The multiple zeta
values, in short MZVs, are defined by
\[
\zeta(i_{1},\dots,i_{k})=\sum_{n_{1}>\cdots>n_{k}\geq1}\frac{1}{n_{1}^{i_{1}%
}\cdots n_{k}^{i_{k}}},
\]
$k\geq1$, for positive integers $i_{1},\dots,i_{k}$ with $i_{1}>1$. This
notation can be extended to alternating, sometimes also called colored
multiple zeta values, by putting a bar over those exponents with an associated
sign in the numerator, as in
\[
\zeta(\bar{5},\bar{2},1)=\sum_{n_{1}>n_{2}>n_{3}\geq1}\frac{(-1)^{n_{1}+n_{2}%
}}{n_{1}^{5}n_{2}^{2}n_{3}}.
\]
Note that $\zeta(i_{1},i_{2},\dots,i_{k})$ converges unless $i_{1}$ is an
unbarred 1. We have the special values $\zeta(\bar{1})=-\log2$ and also
\[
\zeta(\bar{n})=(2^{1-n}-1)\zeta(n)
\]
for $n\geq2$. Multiple zeta values and their variants have been extensively
studied and turned into a vast research area, dating back to Euler and in
modern times started by Hoffman~\cite{H1992} and Zagier~\cite{Z1992}; we also
refer the reader to Borwein et al.~\cite{BBB}. It is often of interest to
determine whether mathematical objects can be written entirely using ordinary
- single argument - zeta values. We will obtain a complete asymptotic
expansion of~\eqref{eq:int1} in terms of a special function. The cases
$q\in\{-1,1\}$ are of special interest and we obtain an expression directly in
terms of alternating multiple zeta values and ordinary MZVs. Furthermore, we
discuss for $q=1$ reducibility to ordinary zeta values. A third motivation for
studying such families comes from random variables and norms. It is classical
that the $p$-norm of a vector converges for $p\rightarrow\infty$ to the
maximum norm. The same questions for random variables, such as
\begin{equation}
Z_{n}=\Vert(U,1-U)\Vert_{n},\quad U=\text{Uniform}[0,1],\label{def:norm}%
\end{equation}
and their asymptotics for $n\rightarrow\infty$ has intriguing relations to
multiple zeta values and their variants~\cite{HKLL2020}. The study of the
moments leads directly to integrals of the form~\eqref{eq:int1}. For example,
the expectation of $Z_{n}$ can be written in terms of $I_{n}$ with $q=w=1$ and
$f(u)=2(1+u)^{-3}$. Higher moments lead to values $w>1$. We will show how the
previous studies~\cite{HKLL2020} can be put under the umbrella
of~\eqref{eq:int1} and can be greatly extended. Therein, the aforehand mentioned special case of $I_n=\E(Z_n)$ was discussed in detail and the asymptotic expansion in terms of single-valued zeta functions was obtained. Here, we obtain a conceptual explanation of this phenomenon for a broad family of integrals, with arbitrary $f$ and general values of $q$ and $w$. This allows us to pinpoint exactly which part of the integral $I_n$ gives rise to expressions related to special functions and multiple zeta values and which part of $I_n$ governs the structure of the coefficients in the asymptotic expansions. 
%
%
Furthermore, a criterion for reduction to single-valued zeta functions is given, also connecting our work to Appell sequences of polynomials. Moreover, the generality of the integral $I_n$ allows to cover a great many concrete examples, as outlined in the last section of this work.

\smallskip

Our main interest is to obtain the asymptotic expansion of $I_{n}$ for $n
\to\infty$ of the form
\[
I_{n} \sim\sum_{p=0}^{\infty} \frac{a_{p}}{n^{p}},
\]
with coefficients $a_{p}=a_{p}(q,w)$. We introduce an important special
function: Nielsen~\cite{Nielsen} defined and studied the generalized
polylogarithm functions $S_{m,p}$, given by the following integral
\begin{equation}
S_{m,p}(z) := \frac{(-1)^{m+p-1}}{(m-1)! p!}\int_{0}^{1} \log^{m-1}(t)
\log^{p}(1-zt) \frac{dt}{t},
\end{equation}
where $m$ and $p$ are positive integers. Note that $S_{m-1,1}=\Li_{m}$, where
the classical $m$-th polylogarithm is defined by the series $\Li_{m}%
(z)=\sum_{k\ge1}\frac{z^{k}}{k^{m}}$. For more properties of $S_{m,p}$, we
refer the reader to~\cite{Kolbig} and also to the recent study of Charlton,
Gangl and Radchenko~\cite{Charlton}.

\smallskip

In the following, we will show that the coefficients $a_{p}=a_{p}(q,w)$ are
always linear combinations of Nielsen's generalized polylogarithm
$S_{m,p}(-q)$. For $q=-1$ it will turn out that $a_{p}(-1,w)$ can always be
expressed as polynomials in ordinary Riemann zeta values, $\zeta(s)$.
Conversely, the coefficients $a_{p}(1,w)$ involve Nielsen's generalized
polylogarithms evaluated at $z=-1$. We prove that in the important special
case $q=w=1$ the coefficients $a_{p}(1,1)$ reduce to polynomials in ordinary
zeta values $\zeta(s)$ if and only if the coefficients of the asymptotic
series of the integral
\[
\int_{0}^{1}f(u)u^{n}du,\quad n\rightarrow\infty
\]
satisfy a certain symmetry condition. Then, in Section~\ref{sec:Apps} we turn
to applications for norms of random variables.

\section{Derivation of the asymptotic series}

\subsection{Logarithms and Stirling numbers}

We write $(1+qu^{n})^{w/n}$ using the $\exp$-$\ln$ representation and expand
the exponential in series.
\begin{align*}
I_{n}  &  =\int_{0}^{1}f(u)\,(1+q\cdot u^{n})^{w/n}\,du\\
&  =\int_{0}^{1}f(u)\left(  1+\sum_{k=1}^{\infty}\frac{1}{k!}\left(  \frac
{w}{n}\log(1+q\cdot u^{n})\right)  ^{k}\right)  du.
\end{align*}
We obtain further $I_{n}=\int_{0}^{1}f(u)\,du+\sum_{k=1}^{\infty}\frac{w^{k}%
}{k!n^{k}}\int_{0}^{1}f(u)\log^{k}(1+qu^{n})du.$ Next we apply the standard
expansion of
\begin{equation}
\log^{k}(1+qx)=k!\sum_{m=0}^{\infty}(-1)^{m-k}%
\genfrac{[}{]}{0pt}{}{m}{k}%
\frac{q^{m}x^{m}}{m!}, \label{eq:ExpanLog}%
\end{equation}
where$%
\genfrac{[}{]}{0pt}{}{m}{k}%
$ denote the unsigned Stirling number of the first kind, also called Stirling
cycle numbers. This gives the exact expression
\[
I_{n}=\int_{0}^{1}f(u)\,du+\sum_{k=1}^{\infty}\frac{(-w)^{k}}{n^{k}}%
\sum_{m=1}^{\infty}%
\genfrac{[}{]}{0pt}{}{m}{k}%
\frac{(-q)^{m}}{m!}\int_{0}^{1}f(u)u^{nm}\,du,
\]
valid for $|w|<n$, where we have interchanged summation and integration.

\subsection{Moments and Watson's lemma}

In order to gain more insight into $I_{n}$, we define the values $\varphi_{r}$
in terms of $f$:
\[
\varphi_{r}=\int_{0}^{1}f(u)u^{r}du,\quad r\in\mathbb{N}.
\]
The values $\varphi_{r}$ can be interpreted probabilistically, subject to
$\int_{0}^{1}f(u)du=1$ and $f$ being non-negative, in other words $f$ being a
density function supported on the unit interval. Then, $\varphi_{r}$ is simply
the $r$-th moment of the corresponding distribution.
%
We rewrite $I_{n}$ using the moments $\varphi_{r}$ to get
\begin{equation}
I_{n}=\int_{0}^{1}f(u)\,du+\sum_{k=1}^{\infty}\frac{(-w)^{k}}{n^{k}}%
\sum_{m=1}^{\infty}%
\genfrac{[}{]}{0pt}{}{m}{k}%
\frac{(-q)^{m}}{m!}\,\varphi_{nm}. \label{eq:Mom}%
\end{equation}
Next we turn to the asymptotics of the moments.

\begin{prop}
\label{Prop:asympMom} The moments $\varphi_{r}=\int_{0}^{1}f(u)u^{r}$ satisfy
the asymptotic expansion
\[
\varphi_{r}\sim\sum_{\nu=0}^{\infty}\frac{\beta_{\nu}}{r^{\nu+1}},\qquad
\beta_{\nu}=(-1)^{\nu}\sum_{\ell=0}^{\nu}%
\genfrac{\{}{\}}{0pt}{}{\nu+1}{\ell+1}%
f^{(\ell)}(1),
\]
where $%
\genfrac{\{}{\}}{0pt}{}{\nu}{\ell}%
$ denote the Stirling numbers of the second kind, also called Stirling
partition numbers.
\end{prop}

\begin{remark}
In some special cases it is possible to turn the asymptotic expansion into an
\emph{exact} identity, compare with~\cite{HKLL2020}.
\end{remark}

Before we turn to the proof, we recall Watson's lemma for Laplace-type integrals.

\begin{lem}
[Watson's Lemma~\cite{Watson}]Suppose $g(t)$ is absolutely integrable on
$[0,\infty)$:
\[
\int_{0}^{\infty}|g(t)|\,dt<\infty.
\]
Suppose further that $g(t)$ is real-analytic at $t=0$, with
\[
g(t)=\sum_{n=0}^{\infty}\frac{g^{(n)}(0)}{n!}t^{n},
\]
then the exponential integral
\[
F(r):=\int_{0}^{\infty}g(t)e^{-rt}\,dt
\]
is finite for all $r>0$ and it has for $r\rightarrow\infty$ the asymptotic
expansion
\begin{equation}
F(r)\sim\sum_{n=0}^{\infty}\frac{g^{(n)}(0)}{r^{n+1}}.
\end{equation}

\end{lem}

\begin{proof}
[Proof of Proposition~\ref{Prop:asympMom}]We use the substitution $u=e^{-t}$
to obtain a Laplace-type integral:
\[
\varphi_{r}=\int_{0}^{\infty}g(t)e^{-rt}\,dt,\quad g(t):=e^{-t}f(e^{-t}).
\]
Watson's lemma yields
\[
\varphi_{r}\sim\sum_{\nu=0}^{\infty}\frac{g^{(\nu)}(0)}{r^{\nu+1}}.
\]
In order to find a simple expression for $g^{(\nu)}(0)$ we use an operator
formula. Let $\partial_{z}$ denote the differential operator with respect to
$z$, $\theta_{z}=z\partial_{z}$ the theta or homogeneity differential operator
and $E_{z=c}$ the evaluation operator at $z=c$. We observe that
\[
\partial_{t}e^{-t}f(e^{-t})=E_{x=e^{-t}}(-\theta_{x})xf(x),
\]
such that by induction
\[
g^{(\nu)}(t)=E_{x=e^{-t}}(-1)^{\nu}\theta_{x}^{\nu}xf(x).
\]
Next we use the expansion
\[
\theta_{x}^{\nu}=\sum_{k=0}^{\nu}%
\genfrac{\{}{\}}{0pt}{}{\nu}{k}%
x^{k}\partial_{x}^{k}%
\]
to obtain
\[
g^{(\nu)}(t)=(-1)^{\nu}\sum_{k=0}^{\nu}%
\genfrac{\{}{\}}{0pt}{}{\nu+1}{k+1}%
e^{-(k+1)t}f^{(k)}(e^{-t}),
\]
where we have used the basic recurrence relation
\[%
\genfrac{\{}{\}}{0pt}{}{\nu+1}{k}%
=k%
\genfrac{\{}{\}}{0pt}{}{\nu}{k}%
+%
\genfrac{\{}{\}}{0pt}{}{\nu}{k-1}%
,\quad0<k<\nu,
\]
for the Stirling numbers of the second kind. Finally, we evaluate at $t=0$ to
obtain the stated identity.
\end{proof}

\subsection{Nielsen's polylogarithm and an asymptotic series}

We continue by applying the asymptotics of the moment $\varphi_{nm}$, as
obtained in Proposition~\ref{Prop:asympMom} to~\eqref{eq:Mom}:
\[
I_{n}\sim\int_{0}^{1}f(u)\,du+\sum_{k=1}^{\infty}\frac{(-w)^{k}}{n^{k}%
}\sum_{m=1}^{\infty}%
\genfrac{[}{]}{0pt}{}{m}{k}%
\frac{(-q)^{m}}{m!}\sum_{\nu=1}^{\infty}\frac{\beta_{\nu-1}}{(nm)^{\nu}}.
\]
This implies that
\[
I_{n}\sim\int_{0}^{1}f(u)\,du+\sum_{k=1}^{\infty}\frac{w^{k}(-1)^{k}}{n^{k}%
}\sum_{\nu=1}^{\infty}\frac{\beta_{\nu-1}}{n^{\nu}}\sum_{m=1}^{\infty}%
\genfrac{[}{]}{0pt}{}{m}{k}%
\frac{(-q)^{m}}{m!m^{\nu}}.
\]
It remains to relate the coefficients to Nielsen's generalized polylogarithm.
The key is the following result.

\begin{lem}
\label{lem:Nielsen} If $|z|\leq1$, then Nielsen's generalized polylogarithm
function satisfies
\[
S_{m,p}(z)=\sum_{j=p}^{\infty}\left[  {%
\genfrac{}{}{0pt}{}{j}{p}%
}\right]  \frac{z^{j}}{j!\,j^{m}}=\sum_{1\leq j_{p}<\cdots<j_{2}<j_{1}}%
\frac{z^{j_{1}}}{j_{1}^{m+1}j_{2}\cdots j_{p}}.
\]

\end{lem}


\begin{proof}
Recall the integral representation
\[
S_{m,p}(z)=\frac{1}{(m-1)!\,p!}\int_{0}^{1}(-\log t)^{m-1}\bigl(-\log
(1-zt)\bigr)^{p}\,\frac{dt}{t}.
\]
For $|z|<1$, the generating function of the unsigned Stirling numbers of the
first kind gives
\[
\frac{\bigl(-\log(1-zt)\bigr)^{p}}{p!}=\sum_{j=p}^{\infty}\left[  {%
\genfrac{}{}{0pt}{}{j}{p}%
}\right]  \frac{(zt)^{j}}{j!}.
\]
Substituting this expansion into the integral and interchanging summation and
integration, we obtain
\[
S_{m,p}(z)=\frac{1}{(m-1)!}\sum_{j=p}^{\infty}\left[  {%
\genfrac{}{}{0pt}{}{j}{p}%
}\right]  \frac{z^{j}}{j!}\int_{0}^{1}(-\log t)^{m-1}t^{\,j-1}\,dt.
\]
Now
\[
\int_{0}^{1}(-\log t)^{m-1}t^{\,j-1}\,dt=\frac{(m-1)!}{j^{m}},
\]
and therefore
\[
S_{m,p}(z)=\sum_{j=p}^{\infty}\left[  {%
\genfrac{}{}{0pt}{}{j}{p}%
}\right]  \frac{z^{j}}{j!\,j^{m}}.
\]

Next we express the Stirling numbers in terms of truncated multiple zeta
values. From
\[
\sum_{k=1}^{j}\left[  {%
\genfrac{}{}{0pt}{}{j}{k}%
}\right]  x^{k}=x(x+1)\cdots(x+j-1)=x(j-1)!\prod_{r=1}^{j-1}\left(  1+\frac
{x}{r}\right)  ,
\]
we get
\[
\left[  {%
\genfrac{}{}{0pt}{}{j}{p}%
}\right]  =(j-1)!\,e_{p-1}\!\left(  1,\frac{1}{2},\dots,\frac{1}{j-1}\right)
,
\]
where $e_{p-1}$ denotes the $(p-1)$-st elementary symmetric polynomial. Hence
\[
\left[  {%
\genfrac{}{}{0pt}{}{j}{p}%
}\right]  =(j-1)!\,\zeta_{j-1}(\{1\}_{p-1}),
\]
where
\[
\zeta_{N}(i_{1},\dots,i_{r}):=\sum_{N\geq n_{1}>\cdots>n_{r}\geq1}\frac
{1}{n_{1}^{\,i_{1}}\cdots n_{r}^{\,i_{r}}}%
\]
denotes the truncated multiple zeta value, with the convention $\zeta
_{N}(\{1\}_{0})=1$.

Substituting this identity into the previous series yields
\[
S_{m,p}(z) =\sum_{j=1}^{\infty}\frac{z^{j}}{j^{m+1}}\, \zeta_{j-1}
(\{1\}_{p-1}).
\]
Finally, by the definition of truncated multiple zeta values,
\[
\zeta_{j_{1}-1}(\{1\}_{p-1}) =\sum_{1\le j_{p}<\cdots<j_{2}<j_{1}}\frac
{1}{j_{2}\cdots j_{p}}.
\]
Therefore
\[
S_{m,p}(z) =\sum_{j_{1}=1}^{\infty}\frac{z^{j_{1}}}{j_{1}^{m+1}} \sum_{1\le
j_{p}<\cdots<j_{2}<j_{1}}\frac{1}{j_{2}\cdots j_{p}} =\sum_{1\le j_{p}
<\cdots<j_{2}<j_{1}} \frac{z^{j_{1}}}{j_{1}^{m+1}j_{2}\cdots j_{p}},
\]
as claimed.

The cases $z=\pm1$ follow by absolute convergence.
\end{proof}



By Lemma~\ref{lem:Nielsen} we obtain
\begin{equation}
S_{m,p}(1)=\zeta(m+1,\{1\}_{p-1}),\quad S_{m,p}(-1)=\zeta(\overline
{m+1},\{1\}_{p-1}). \label{eqn:NielsenAMZV}%
\end{equation}
Concerning $I_{n}$, we get further
\[
I_{n}\sim\int_{0}^{1}f(u)\,du+\sum_{\nu=1}^{\infty}\beta_{\nu-1}\sum
_{k=1}^{\infty}(-w)^{k}\frac{S_{\nu,k}(-q)}{n^{\nu+k}}.
\]

By grouping powers of $\frac{1}{n}$, we obtain our first main result, namely
an asymptotic series for $I_{n}$.

\begin{theorem}
\label{the:main} The integral $I_{n}$ has the following asymptotic series for
$n\to\infty$
\[
I_{n} \sim\int_{0}^{1} f(u)\,du + \sum_{p=2}^{\infty} \frac{a_{p}}{n^{p}},
\]
where the coefficients $a_{p}=a_{p}(q,w)$ are given in terms of $f$, Nielsen's
polylogarithm and Stirling numbers of the second kind:
\[
a_{p}=\sum_{\ell=1}^{p-1} (-w)^{p-\ell}\beta_{\ell-1} 
S_{\ell,p-\ell}(-q), \quad\beta_{\nu}=(-1)^{\nu} \sum_{\ell=0}^{\nu}%
\genfrac{\{}{\}}{0pt}{}{\nu+1}{\ell+1}
f^{(\ell)}(1).
\]

In the special case $q=-1$ the coefficients $a_{p}$ reduce to multiple zeta
values and also ordinary zeta values,
\[
a_{p}=\sum_{\ell=1}^{p-1}(-w)^{p-\ell} \beta_{\ell-1} \zeta
(\ell+1,\{1\}_{p-\ell-1}),
\]
whereas for $q=1$ the coefficients involve alternating multiple zeta values:
\[
a_{p}=\sum_{\ell=1}^{p-1}(-w)^{p-\ell} \beta_{\ell-1} 
\zeta(\overline{\ell+1},\{1\}_{p-\ell-1}),
\]

\end{theorem}

Borwein, Bradley and Broadhurst proved that for all positive integers $n,m$
the multiple zeta value $\zeta(m+1,\{1\}_{n})$ is a rational polynomial in the
$\zeta(i)$~\cite[Eq. (10)]{BBB}:
\begin{equation}
\sum_{m,n\geq0}\zeta(m+2,\{1\}_{n})x^{m+1}y^{n+1}=1-\exp\biggl(\sum_{k\geq
2}\frac{x^{k}+y^{k}-(x+y)^{k}}{k}\zeta(k)\biggr). \label{eqn:BBB}%
\end{equation}
Alternatively, K\"{o}lbig~\cite{Kolbig} gave a recurrence relation for the
values $\zeta(m+1,\{1\}_{n})$ in terms of single-value zetas.
Thus, the coefficients $a_{p}(-1,w)$ are always reducible to polynomials in
ordinary zeta values. We discuss later the reducibility of the case $q=1$ and
$a_{p}=a_{p}(1,w)$. Before, we turn to simplifications of $\beta_{\nu}$.

\subsection{Appell sequences}

Next we study in more detail the coefficients $\beta_{\nu}$ of the moments
$\varphi_{r}=\int_{0}^{1}f(u)u^{r}$. An Appell-type family $\{P_{n}%
(x)\}_{n=0}^{\infty}$ of polynomials is defined by an exponential generating
function of the form
\begin{equation}
e^{xt}\ap(t)=\sum_{n=0}^{\infty}P_{n}(x)\frac{t^{n}}{n!},
\end{equation}
where $\ap(t)$ is analytic at $t=0$. If, in addition, $\ap%
(0)\neq0$, then $\{P_{n}(x)\}_{n=0}^{\infty}$ is a normalized Appell sequence
in the classical sense. The function $\ap(t)$ is called the
\emph{Appell seed}.

More generally, even when $\ap(0)=0$, the generating function still
defines a polynomial family, and the standard Appell identities remain valid,
in particular,
\[
\frac{d}{dx}P_{n}(x)=nP_{n-1}(x),\quad n\in\mathbb{N},
\]
as well as the addition theorem:
\begin{equation}
P_{n}(x+y)=\sum_{k=0}^{n}\binom{n}{k}P_{k}(x)y^{\,n-k}.
\label{eqn:AppellAdditionThm}%
\end{equation}
For this reason, in the sequel, we also allow such Appell-type families, which
include, for example, the generalized Genocchi polynomials.

We also mention the reflection symmetry: if the Appell seed $\ap(t)$
satisfies $\ap(-t)=e^{\omega t}\ap(t)$, then
\begin{equation}
P_{n}(\omega-x)=(-1)^{n}P_{n}(x). \label{eqn:AppellReflect}%
\end{equation}
Classical examples (beyond the trivial $P_{n}(x)=x^{n}$) include the Hermite
polynomials, the Bernoulli polynomials, and the generalized Euler polynomials
$E_{m}^{(d)}(x)$, with Appell seed
\[
\ap(t)=\left(  \frac{2}{1+e^{t}}\right)  ^{d}.
\]
Next we show that if the function $f(u)$ is related to an Appell-type seed,
then the coefficients $\{\beta_{\nu}\}$ are given by the corresponding
polynomial family evaluated at a constant.

\begin{theorem}
\label{Prop:Appell} Let $f(u)$ be written as
\begin{equation}
f(u)=b\cdot\ap(c\ln u), \label{eqn:f}%
\end{equation}
where $b,c\neq0$ are real non-zero constants and $\ap(t)$ is the seed
of the Appell-type sequence $\{P_{n}(x)\}_{n=0}^{\infty}$. Then the
coefficients $\beta_{\nu}$ of the asymptotic expansion of $\varphi_{r}%
=\int_{0}^{1}f(u)u^{r}du$ are given by
\begin{equation}
\beta_{\nu}=b\cdot(-1)^{\nu}c^{\nu}P_{\nu}\!\left(  \frac{1}{c}\right)  .
\end{equation}

\end{theorem}

\begin{proof}
Let again $g(t)=e^{-t}f(e^{-t})$. By our assumption on $f$ we get
\[
g(t)=e^{-t}f(e^{-t})=be^{-t}\ap(-ct).
\]
Now recall that the Appell-type family $\{P_{n}(x)\}$ generated by the seed
$\ap(t)$ is defined via the exponential generating function
\[
\ap(t)e^{xt}=\sum_{\nu=0}^{\infty}P_{\nu}(x)\frac{t^{\nu}}{\nu!}.
\]
This implies that
\[
g(t)=\sum_{\nu=0}^{\infty}bP_{\nu}(\frac{1}{c})\frac{(-ct)^{\nu}}{\nu!}%
\]
On the other hand, as we know that
\[
g(t)=\sum_{\nu=0}^{\infty}g^{(\nu)}(0)\frac{t^{\nu}}{\nu!}=\sum_{\nu
=0}^{\infty}\beta_{\nu}\frac{t^{\nu}}{\nu!},
\]
which implies the stated result.
\end{proof}

Below we collect a few examples.

\begin{example}
[Monomials]The seed of the monomials $\{x^{n}\}$ is $\ap(t)=1$. If
$f(u)=1$, then $\beta_{\nu}=(-1)^{\nu}$.
\end{example}

\begin{example}
[Generalized Bernoulli polynomials]The seed of the generalized Bernoulli
polynomials $\{B_{n}^{(d)}(x)\}$ is given by
\[
\ap(t)=\left(  \frac{t}{e^{t}-1}\right)  ^{d}.
\]
If $f(u)=\ln^{d}(u)/(u^{c}-1)^{d}$, then
\[
f(u)=\frac{1}{c^{d}}\ap(c\ln(u)),\quad\text{and}\quad\beta_{\nu}%
=\frac{1}{c^{d}}(-1)^{\nu}B_{\nu}^{(d)}\!\left(  \frac{1}{c}\right)  .
\]

\end{example}

\begin{example}
[Generalized Euler polynomials]The seed of the generalized Euler polynomials
$\{E_{n}^{(d)}(x)\}$ is given by
\[
\ap(t)=\left(  \frac{2}{1+e^{t}}\right)  ^{d}.
\]
If $f(u)=1/(1+u^{c})^{d}$, then
\[
f(u)=\frac{1}{2^{d}}\ap(c\ln(u)),\quad\text{and}\quad\beta_{\nu}%
=\frac{1}{2^{d}}(-1)^{\nu}c^{\nu}E_{\nu}^{(d)}\!\left(  \frac{1}{c}\right)  .
\]

\end{example}

\begin{example}
[Generalized Genocchi polynomials]The seed of the generalized Genocchi
polynomials $\{G_{n}^{(d)}(x)\}$ satisfies
\[
\ap(t)=\left(  \frac{2t}{1+e^{t}}\right)  ^{d}.
\]
If $f(u)=\ln^{d}(u)/(1+u^{c})^{d}$, then
\[
f(u)=\frac{1}{(2c)^{d}}\ap(c\ln(u)),\quad\text{and}\quad\beta_{\nu
}=\frac{1}{(2c)^{d}}(-1)^{\nu}c^{\nu}G_{\nu}^{(d)}\!\left(  \frac{1}%
{c}\right)  .
\]

\end{example}

\begin{example}
[Probabilist's Hermite polynomials]The seed of the probabilist's Hermite
polynomials $\{He_{n}(x)\}$ is $\ap(t)=e^{-\frac{t^{2}}{2}}$. If
$f(u)=e^{-\frac{c^{2}\ln^{2}(u)}{2}}$, then
\[
f(u)=\ap(c\ln(u)),\quad\text{and}\quad\beta_{\nu}=(-1)^{\nu}c^{\nu
}He_{\nu}\!\left(  \frac{1}{c}\right)  .
\]

\end{example}

\section{MZVs and alternating MZVs}

\subsection{A symmetry condition and reduction to non-alternating MZVs}

In the evaluation of the integral $I_{n}$ the case $q=1$ is of special
interest~\cite{HKLL2020}. We consider
\[
I_{n}=\int_{0}^{1}f(u)(1+u^{n})^{w/n}\,du\sim\int_{0}^{1}f(u)\,du+\sum
_{p=2}^{\infty}\frac{a_{p}}{n^{p}}.
\]
Here, the values $a_{p}=a_{p}(1,w)$ are given by
\begin{equation}
a_{p}=\sum_{\ell=1}^{p-1}(-w)^{p-\ell}\beta_{\ell-1}S_{\ell,p-\ell
}(-1)=\sum_{\ell=1}^{p-1}w^{p-\ell}\beta_{\ell-1}\sigma_{\ell,p-\ell},
\label{eqn:QuOne0}%
\end{equation}
where we have used K\"{o}lbig's notation for the special values of Nielsen's
polylogarithm, reducing to alternating MZVs~\eqref{eqn:NielsenAMZV}:
\[
S_{n,p}(1)=s_{n,p},\qquad(-1)^{p}S_{n,p}(-1)=\sigma_{n,p}.
\]
Thus, the coefficients $a_{p}=a_{p}(1,w)$ depend on the $\sigma_{j,k}$ values.
We obtain the following reduction result, subject to a symmetry condition on a
weighted binomial transform of the coefficients $\beta_{\nu}$.



\begin{theorem}
Assume the coefficients $\beta_{\nu}$, defined in terms of $f$ by
\[
\beta_{\nu}=(-1)^{\nu}\sum_{\ell=0}^{\nu}%
\genfrac{\{}{\}}{0pt}{}{\nu+1}{\ell+1}%
f^{(\ell)}(1),
\]
satisfy for all $p\geq2$ the conditions
\[
\sum_{\ell=1}^{\nu}(-1)^{\ell-\nu}\binom{\nu-1}{\ell-1}\,w^{\ell}\beta
_{p-\ell-1}=\sum_{\ell=1}^{p-\nu}(-1)^{\ell-p+\nu}\binom{p-\nu-1}{\ell
-1}\,w^{\ell}\beta_{p-\ell-1},
\]
for $1\leq\nu\leq\left\lfloor \frac{p}{2}\right\rfloor$. Then, the integral $I_{n}$ has the expansion
\[
I_{n}=\int_{0}^{1}f(u)(1+u^{n})^{w/n}\,du\sim\int_{0}^{1}f(u)\,du+\sum
_{p=2}^{\infty}\frac{1}{n^{p}}\sum_{\nu=1}^{p-1}\rho_{p,\nu}\,s_{\nu,p-\nu},
\]
where one symmetric choice of coefficients is
\[
\rho_{p,\nu}=\frac{1}{2}\sum_{\ell=1}^{\nu}(-1)^{\nu-\ell}\binom{\nu-1}%
{\ell-1}\,w^{\ell}\beta_{p-\ell-1},
\]
with $1\leq\nu\leq\left\lfloor \frac{p}{2}\right\rfloor $, together with
$\rho_{p,p-\nu}=\rho_{p,\nu}$.
\end{theorem}

\begin{proof}
K\"olbig~\cite[Theorem 3]{Kolbig} provides a relation connecting $\sigma
_{j,k}$ and $s_{j,k}$, which reads
\begin{equation}
\label{eq:KolbigIdent}\sum_{\nu=1}^{j}\binom{j+k-\nu-1}{k-1}\sigma
_{\nu,j+k-\nu} +\sum_{\nu=1}^{k}\binom{j+k-\nu-1}{j-1}\sigma_{\nu,j+k-\nu}
=s_{j,k}.
\end{equation}

We seek coefficients $\rho_{p,\ell}$ to express $a_{p}$ as a linear
combination of the $s_{j,k}$ values. Rewrite~\eqref{eq:KolbigIdent} by setting
$p:=j+k$:
\[
\sum_{\nu=1}^{j}\binom{p-\nu-1}{p-j-1}\sigma_{\nu,p-\nu} +\sum_{\nu=1}%
^{p-j}\binom{p-\nu-1}{j-1}\sigma_{\nu,p-\nu} =s_{j,p-j}.
\]
Extending the range of summation, including the zero values of the binomials,
gives
\begin{equation}
\sum_{\nu=1}^{p-1}\sigma_{\nu,p-\nu} \left(  \binom{p-\nu-1}{p-j-1}%
+\binom{p-\nu-1}{j-1} \right)  =s_{j,p-j},\qquad1\le j\le p-1.
\label{eqn:QuOne1}%
\end{equation}

A linear combination of~\eqref{eqn:QuOne1} is
\[
\sum_{\ell=1}^{p-1}\rho_{p,\ell} \sum_{\nu=1}^{p-1}\sigma_{\nu,p-\nu} \left(
\binom{p-\nu-1}{p-\ell-1}+\binom{p-\nu-1}{\ell-1} \right)  = \sum_{\ell
=1}^{p-1}\rho_{p,\ell}\,s_{\ell,p-\ell}.
\]
By change of order of summation, we get
\[
\sum_{\nu=1}^{p-1}\sigma_{\nu,p-\nu} \sum_{\ell=1}^{p-1}\rho_{p,\ell} \left(
\binom{p-\nu-1}{p-\ell-1}+\binom{p-\nu-1}{\ell-1} \right)  = \sum_{\ell
=1}^{p-1}\rho_{p,\ell}\,s_{\ell,p-\ell}.
\]

If we find $(\rho_{p,\ell})_{\ell=1}^{p-1}$ such that
\begin{equation}
\sum_{\ell=1}^{p-1}\rho_{p,\ell} \left(  \binom{p-\nu-1}{p-\ell-1}%
+\binom{p-\nu-1}{\ell-1} \right)  = w^{p-\nu}\beta_{\nu-1}, \qquad1\le\nu\le
p-1, \label{eqn:QuOne2}%
\end{equation}
then a combination of~\eqref{eqn:QuOne2} and~\eqref{eqn:QuOne0} gives the
desired expression:
\[
\sum_{\nu=1}^{p-1}w^{p-\nu}\beta_{\nu-1}\sigma_{\nu,p-\nu} = \sum_{\ell
=1}^{p-1}\rho_{p,\ell}\,s_{\ell,p-\ell}.
\]

Now our goal is to find a sufficient condition on the values $\beta_{p}$ and
$w$ such that the system~\eqref{eqn:QuOne2} is solvable. Re-indexing
\eqref{eqn:QuOne2}, we obtain for $p\ge2$ the equivalent system
\begin{equation}
\sum_{\ell=1}^{p-1}\binom{p-\nu-1}{\ell-1} \left(  \rho_{p,\ell}%
+\rho_{p,p-\ell} \right)  = w^{p-\nu}\beta_{\nu-1}, \qquad1\le\nu\le p-1.
\label{eqn:QuOne3}%
\end{equation}
We introduce the pair-sums
\[
\eta_{p,\nu}:=\rho_{p,\nu}+\rho_{p,p-\nu}, \qquad1\le\nu\le p-1.
\]
By their definition, they have to satisfy the symmetry condition
\begin{equation}
\eta_{p,\nu}=\eta_{p,p-\nu}, \qquad\nu=1,\ldots,\left\lfloor \frac{p}%
{2}\right\rfloor , \label{eqn:symm}%
\end{equation}
using the convention $\eta_{p,p/2}=2\rho_{p,p/2}$ when $p$ is even. In terms
of the pair-sums $\eta$ the equations~\eqref{eqn:QuOne3} can be written as
\begin{equation}
\sum_{\ell=1}^{p-1}\binom{p-\nu-1}{\ell-1}\eta_{p,\ell} = w^{p-\nu}\beta
_{\nu-1}, \qquad1\le\nu\le p-1. \label{eqn:QuOne4}%
\end{equation}
It is straightforward to solve the Pascal matrix system~\eqref{eqn:QuOne4} by
binomial inversion. We arrive at the solution
\begin{equation}
\eta_{p,\nu} = \sum_{j=1}^{\nu} (-1)^{\nu-j}\binom{\nu-1}{j-1}\,w^{j}%
\beta_{p-j-1}, \qquad1\le\nu\le p-1. \label{eqn:sol}%
\end{equation}
Now a condition for the system to have a solution is consistency: the
$\eta_{p,\nu}$ values~\eqref{eqn:sol} have to satisfy the solvability
condition \eqref{eqn:symm}. This translates into the constraints
\[
\sum_{\ell=1}^{\nu} (-1)^{\ell-\nu}\binom{\nu-1}{\ell-1}\,w^{\ell}%
\beta_{p-\ell-1} = \sum_{\ell=1}^{p-\nu} (-1)^{\ell-p+\nu}\binom{p-\nu-1}%
{\ell-1}\,w^{\ell}\beta_{p-\ell-1}, 
\]
with $1\le\nu\le\left\lfloor \frac{p}{2}\right\rfloor$. Note that once the symmetry condition~\eqref{eqn:symm} holds, the pair-sums
are fixed by~\eqref{eqn:sol}. This solution does not separate the individual
values $\rho_{p,k}$ and $\rho_{p,p-k}$; it fixes only the $\eta_{p,k}$'s. The
differences $\rho_{p,k}-\rho_{p,p-k}$ remain undetermined. We reduce the
dimension of the solution space to $1$ by imposing the symmetry
\[
\rho_{p,\nu}=\rho_{p,p-\nu}, \qquad\nu=1,\ldots,\left\lfloor \frac{p}%
{2}\right\rfloor .
\]
Thus,
\[
\rho_{p,\nu} = \frac{\eta_{p,\nu}}{2} = \frac12 \sum_{\ell=1}^{\nu}
(-1)^{\nu-\ell}\binom{\nu-1}{\ell-1}\,w^{\ell}\beta_{p-\ell-1},
\]
for $1\le\nu\le\left\lfloor \frac{p}{2}\right\rfloor $.
\end{proof}

\subsection{The binomial transform and the symmetry condition}

The solvability condition is a symmetry requirement on a weighted binomial
transform of the coefficients $\beta_{\nu}$. Our next goal is to study this
property in more detail.

Let $(a_{\nu})$ be a sequence of numbers and define its binomial
transform~\cite{Knuth} $(\psi_{\nu})$ by:
\begin{equation}
\psi_{\nu}=\sum_{\ell=0}^{\nu}(-1)^{\ell}\binom{\nu}{\ell}a_{\ell}.
\end{equation}
We seek a condition on the sequence $(a_{\nu})$ such that the transformed
sequence satisfies the following symmetry for a fixed integer $p$:
\begin{equation}
\psi_{\nu}=(-1)^{p}\psi_{p-\nu},\qquad0\leq\nu\leq p.
\end{equation}
To find the necessary and sufficient condition, we utilize the method of
generating polynomials. Let us define a generating polynomial $\Psi_{p}(z)$
for the first $p+1$ terms of the sequence $(\psi_{\nu})$:
\begin{equation}
\Psi_{p}(z)=\sum_{\nu=0}^{p}\binom{p}{\nu}\psi_{\nu}z^{\nu}.
\end{equation}
Similarly, we define a generating polynomial $A_{p}(x)$ based on the sequence
$(a_{\nu})$:
\begin{equation}
A_{p}(x)=\sum_{k=0}^{p}\binom{p}{k}(-1)^{k}a_{k}x^{k}. \label{def:genpoly}%
\end{equation}

\begin{theorem}
\label{the:symmPoly} Let $(a_{\nu})$ be a sequence of numbers and let
$(\psi_{\nu})$ denote its binomial transform. The sequence $(\psi_{\nu})$
satisfies the symmetry
\[
\psi_{\nu}= (-1)^{p} \psi_{p-\nu},\quad0 \le\nu\le p,
\]
if and only if the generating polynomial
\[
A_{p}(x) = \sum_{k=0}^{p} \binom{p}{k} (-1)^{k} a_{k} x^{k}%
\]
satisfies the symmetry condition:
\[
A_{p}(x) = (-1)^{p} A_{p}(1-x).
\]

\end{theorem}

This implies that $A_{p}(x)$ is symmetric (if $p$ is even) or anti-symmetric
(if $p$ is odd) about the line $x=\frac{1}{2}$.

\begin{proof}
First, we express the symmetry condition in terms of the polynomial $\Psi
_{p}(z)$. Substituting $\psi_{\nu}=(-1)^{p}\psi_{p-\nu}$ into $\Psi_{p}(z)$:
\[
\Psi_{p}(z)=\sum_{\nu=0}^{p}\binom{p}{\nu}(-1)^{p}\psi_{p-\nu}z^{\nu}.
\]
Let $j=p-\nu$, then
\[
\Psi_{p}(z)=(-1)^{p}z^{p}\sum_{j=0}^{p}\binom{p}{j}\psi_{j}\left(  \frac{1}%
{z}\right)  ^{j}.
\]
Thus, the condition on $(\psi_{\nu})$ is equivalent to the functional
equation:
\begin{equation}
\Psi_{p}(z)=(-1)^{p}z^{p}\Psi_{p}(z^{-1}). \label{eqn:Psi}%
\end{equation}
We now express $\Psi_{p}(z)$ in terms of $(a_{\nu})$. By definition of
$\psi_{\nu}$ we get
\[
\Psi_{p}(z)=\sum_{\nu=0}^{p}\binom{p}{\nu}z^{\nu}\sum_{\ell=0}^{\nu}\binom
{\nu}{\ell}(-1)^{\ell}a_{\ell}.
\]
Furthermore, we have
\[
\Psi_{p}(z)=\sum_{\ell=0}^{p}(-1)^{l}a_{l}\sum_{\nu=l}^{p}\binom{p}{\nu}%
\binom{\nu}{l}z^{\nu}=\sum_{\ell=0}^{p}\binom{p}{\ell}(-1)^{\ell}a_{\ell}%
\sum_{\nu=\ell}^{p}\binom{p-\ell}{\nu-\ell}z^{\nu}.
\]
Further simplifications give
\[
\Psi_{p}(z)=\sum_{\ell=0}^{p}\binom{p}{\ell}(-1)^{\ell}a_{\ell}z^{\ell
}(1+z)^{p-\ell}=(1+z)^{p}\sum_{\ell=0}^{p}\binom{p}{\ell}(-1)^{\ell}a_{\ell
}\left(  \frac{z}{1+z}\right)  ^{\ell}.
\]
Observing the definition of $A_{p}(x)$ in (\ref{def:genpoly}), we arrive at
the fundamental identity:
\begin{equation}
\Psi_{p}(z)=(1+z)^{p}A_{p}\!\left(  \frac{z}{1+z}\right)  .
\end{equation}
We now apply this identity to the symmetry condition~\eqref{eqn:Psi}:
\[
(1+z)^{p}A_{p}\left(  \frac{z}{1+z}\right)  =(-1)^{p}(1+z)^{p}A_{p}\!\left(
\frac{1}{1+z}\right)  .
\]
Dividing by the non-zero factor $(1+z)^{p}$, we obtain:
\[
A_{p}\!\left(  \frac{z}{1+z}\right)  =(-1)^{p}A_{p}\!\left(  \frac{1}%
{1+z}\right)  ,
\]
which leads to the stated condition after setting $x=\frac{z}{1+z}$.
\end{proof}

\subsection{Checking the solvability condition for coefficients related to
Appell-type families}

Recall~from \eqref{eqn:f} that for functions $f(u)=b\cdot\ap(c\ln u)$
the coefficients $\beta_{\nu}$ are given by
\[
\beta_{\nu}=b(-1)^{\nu}c^{\nu}P_{\nu}\!\left(  \frac{1}{c}\right)  ,
\]
where $P_{\nu}(x)$ are the polynomials generated by $\ap(t)$.

For fixed $p\ge2$, define the weighted sequence
\[
a_{k}^{(p,w)}:=w^{k+1}\beta_{p-k-2}, \qquad0\le k\le p-2.
\]
Then~\eqref{eqn:sol} can be rewritten as
\[
\eta_{p,\nu}=(-1)^{\nu-1}\psi_{\nu-1}^{(p,w)}, \qquad1\le\nu\le p-1,
\]
where $(\psi_{m}^{(p,w)})_{m=0}^{p-2}$ is the binomial transform of
$(a_{k}^{(p,w)})_{k=0}^{p-2}$:
\[
\psi_{m}^{(p,w)} = \sum_{\ell=0}^{m}(-1)^{\ell}\binom{m}{\ell}a_{\ell}%
^{(p,w)}.
\]
Hence the symmetry condition $\eta_{p,\nu}=\eta_{p,p-\nu}$ is equivalent to
\[
\psi_{m}^{(p,w)}=(-1)^{p}\psi_{p-2-m}^{(p,w)}, \qquad0\le m\le p-2.
\]
Since $p$ and $p-2$ have the same parity, Theorem~\ref{the:symmPoly} applies
with $N:=p-2$. Therefore this condition holds if and only if the polynomial
\[
A_{p-2}^{(w)}(x) = \sum_{k=0}^{p-2}\binom{p-2}{k}(-1)^{k}a_{k}^{(p,w)}x^{k}
\]
satisfies
\[
A_{p-2}^{(w)}(x)=(-1)^{p}A_{p-2}^{(w)}(1-x).
\]

Using the Appell form of $\beta_{\nu}$, we get
\[
a_{k}^{(p,w)}=w^{k+1}b(-1)^{p-k-2}c^{p-k-2}P_{p-k-2}\!\left(  \frac{1}%
{c}\right)  ,
\]
hence
\[
A_{p-2}^{(w)}(x)=b(-1)^{p-2}wc^{p-2}\sum_{k=0}^{p-2}\binom{p-2}{k}%
P_{p-k-2}\!\left(  \frac{1}{c}\right)  \left(  \frac{wx}{c}\right)  ^{k}.
\]
By the addition theorem for Appell sequences~\eqref{eqn:AppellAdditionThm},
with $n=p-2$, $y=1/c$, and $z=wx/c$, this becomes
\[
A_{p-2}^{(w)}(x)=b(-1)^{p-2}wc^{p-2}P_{p-2}\!\left(  \frac{1+wx}{c}\right)  .
\]
Therefore the solvability condition is equivalent to
\[
P_{p-2}\!\left(  \frac{1+wx}{c}\right)  =(-1)^{p}P_{p-2}\!\left(
\frac{1+w-wx}{c}\right)  .
\]
If the Appell family satisfies the reflection symmetry
\[
P_{n}(z)=(-1)^{n}P_{n}(\omega-z),
\]
then, since $(-1)^{p}=(-1)^{p-2}$, this is equivalent to
\[
\omega-\frac{1+wx}{c}=\frac{1+w-wx}{c},
\]
which yields
\[
c\omega=w+2.
\]
Thus, the solvability condition holds whenever the Appell-type family
possesses a reflection symmetry around $\omega/2$ and
\[
c\omega=w+2.
\]

\subsection{Solvability condition for coefficients related to several
Appell-type families}

\begin{example}
[Generalized Euler polynomials]\label{ex:GenEulerPoly} The Appell seed is
\[
\ap(t):=\left(  \frac{2}{1+e^{t}}\right)  ^{d}.
\]
We verify that $\ap(-t)=e^{dt}\ap(t)$, hence $E_{n}^{\left(
d\right)  }\left(  d-x\right)  =\left(  -1\right)  ^{n}E_{n}^{\left(
d\right)  }\left(  x\right)  .\ $It follows that $\omega=d$. \newline Thus,
the coefficients
\[
\beta_{\nu}=\frac{1}{2^{d}}(-1)^{\nu}c^{\nu}E_{\nu}^{(d)}\!\left(  \frac{1}%
{c}\right)
\]
related to generalized Euler polynomials satisfy the solvability condition if
$cd=w+2$.
\end{example}

\begin{example}
[Generalized Genocchi polynomials]The Appell seed is
\[
\ap(t):=\left(  \frac{2t}{e^{t}+1}\right)  ^{d}.
\]
We verify that $\ap(-t)=(-1)^{d}e^{dt}\ap(t)$, hence $\left(
-1\right)  ^{d}G_{n}^{\left(  d\right)  }\left(  d-x\right)  =\left(
-1\right)  ^{n}G_{n}^{\left(  d\right)  }\left(  x\right)  .$ It follows that
if $d$ is even then $\omega=d$.\newline Thus, the coefficients
\[
\beta_{\nu}=\frac{1}{(2c)^{d}}(-1)^{\nu}c^{\nu}G_{\nu}^{(d)}\!\left(  \frac
{1}{c}\right)
\]
related to generalized Genocchi polynomials satisfy the solvability condition
if $cd=w+2$ and $d$ is even.
\end{example}

\begin{example}
[Generalized Bernoulli polynomials]The Appell seed is
\[
\ap(t):=\left(  \frac{t}{e^{t}-1}\right)  ^{d}.
\]
We verify that $\ap(-t)=e^{dt}\ap(t)$, hence $B_{n}^{\left(
d\right)  }\left(  d-x\right)  =\left(  -1\right)  ^{n}B_{n}^{\left(
d\right)  }\left(  x\right)  .\ $It follows that $\omega=d$.\newline Thus, the
coefficients
\[
\beta_{\nu}=\frac{1}{c^{d}}(-1)^{\nu}B_{\nu}^{(d)}\!\left(  \frac{1}%
{c}\right)
\]
related to generalized Bernoulli polynomials satisfy the solvability condition
if $cd=w+2$.
\end{example}

\begin{example}
[Probabilist's Hermite polynomials]\label{ex:GenHermitePoly} The Appell seed
is
\[
\ap(t):=e^{-\frac{t^{2}}{2}}.
\]
We verify that $\ap(-t)=\ap(t)$, hence $He_{n}\left(
-x\right)  =\left(  -1\right)  ^{n}He_{n}\left(  x\right)  .\ $It follows that
$\omega=0$.\newline Thus, the coefficients
\[
\beta_{\nu}=(-1)^{\nu}c^{\nu}He_{\nu}\left(  \frac{1}{c}\right)
\]
related to probabilist's Hermite polynomials satisfy the solvability condition
if $\ w=-2$.
\end{example}

\section{Applications\label{sec:Apps}}
\subsection{Sine-Cosine integrals}
We relate the integral ~\eqref{eq:int2} to the family ~\eqref{eq:int1}.
Applying the substitution $u=\tan(x)$, the integral~\eqref{eq:int2} transforms
as follows:
\[
\int_{0}^{\frac{\pi}{2}}\cos(x)\bigl|1-\tan^{n}(x)\bigr|^{\frac{1}{n}}%
dx=\int_{0}^{\infty}\frac{\lvert1-u^{n}\rvert^{\frac{1}{n}}}{(1+u^{2}%
)^{\frac{3}{2}}}du.
\]
Utilizing the symmetry of the integrand, this expression simplifies to
\begin{equation}
\int_{0}^{1}\frac{2}{(1+u^{2})^{\frac{3}{2}}}(1-u^{n})^{\frac{1}{n}}du.
\end{equation}
To be more precise, we split the integral
\[
\int_{0}^{\infty}\frac{|1-u^{n}|^{1/n}}{(1+u^{2})^{3/2}}\,du=\int_{0}^{1}%
\frac{(1-u^{n})^{1/n}}{(1+u^{2})^{3/2}}\,du+\int_{1}^{\infty}\frac
{(u^{n}-1)^{1/n}}{(1+u^{2})^{3/2}}\,du.
\]
Using the substitution $u=\tfrac{1}{x}$ in the second integral, we obtain
\[
\int_{1}^{\infty}\frac{(u^{n}-1)^{1/n}}{(1+u^{2})^{3/2}}\,du=\int_{0}^{1}%
\frac{(1-x^{n})^{1/n}}{(1+x^{2})^{3/2}}\,dx.
\]
Therefore,
\[
\int_{0}^{\infty}\frac{|1-u^{n}|^{1/n}}{(1+u^{2})^{3/2}}du=\int_{0}%
^{1}f(u)(1-u^{n})^{1/n}du,\,\text{with }f(u)=\frac{2}{(1+u^{2})^{3/2}}.
\]
Our main result in Theorem~\ref{the:main} and~\eqref{eqn:BBB} leads to a
complete asymptotic expansion in terms of ordinary zeta values, where the
coefficients $\beta_{\nu}$ are determined by $f$ and Theorem~\ref{Prop:Appell}%
, as $f$ can be written using the Appell seed of the generalized Euler
polynomials:
\[
\ap(t)=\left(  \frac{2}{1+e^{t}}\right)  ^{\frac{3}{2}},\quad
f(u)=\frac{2}{(1+u^{2})^{3/2}}=\frac{1}{\sqrt{2}}\cdot\ap%
\big(2\ln(u)\big).
\]
This implies that $\beta_{\nu}$ is given by the generalized Euler
polynomials:
\[
\beta_{\nu}=\frac{2^{\nu}}{\sqrt{2}}(-1)^{\nu}E_{\nu}^{(3/2)}(\frac{1}{2}).
\]
We summarize our findings below
\[%
\begin{split}
I_{n}  &  \sim\int_{0}^{1}\frac{2}{(1+u^{2})^{3/2}}\,du+\sum_{p=2}^{\infty
}\frac{a_{p}}{n^{p}}\\
&  =\sqrt{2}+\sum_{p=2}^{\infty}\frac{1}{n^{p}}\sum_{\ell=1}^{p-1}%
\frac{2^{\ell-1}}{\sqrt{2}}(-1)^{\ell-1}E_{\ell-1}^{(3/2)}(\frac{1}%
{2})(-1)^{p-\ell}\zeta(\ell+1,\{1\}_{p-\ell-1}),
\end{split}
\]
where the first few concrete values of the coefficients are given by
\begin{align*}
a_{2}&=\frac{\sqrt{2}\pi^{2}}{24},\quad
a_{3}=\frac{\sqrt{2}\zeta(3)}{4},\quad a_{4}=-\frac{\sqrt{2}\pi^{4}}{1152},\\
a_{5}  &  =-\frac{\sqrt{2}\pi^{2}\zeta(3)}{32}+\frac{\sqrt{2}\zeta(5)}%
{16},\quad a_{6}=\frac{131\sqrt{2}\pi^{6}}{193536}-\frac{3\sqrt{2}\zeta
(3)^{2}}{32}.
\end{align*}

This integral can be interpreted as a special instance of the moments of the random variable  
\[
T_n = \Big|\sin^n\big(\frac{\pi U}2 \big)-\cos^n\big(\frac{\pi U}2 \big)\Big|^{\frac1n}
\]
with $U$ uniformly distributed on $[0,1]$. The expected value is a constant multiple of the integral treated before:
\[
\E(T_n)=\int_0^{1}\Big|\sin^n\big(\frac{\pi x}2 \big)-\cos^n\big(\frac{\pi x}2 \big)\Big|^{\frac1n}dx
=\frac{2}{\pi}\int_0^{\frac{\pi}2}|\sin^n(u)-\cos^n(u)|^{\frac1n}du.
\]
Moreover, higher moments of $T_n$, $w>0$, lead to the integrals
\[
\E(T_n^w)=\frac{2}{\pi}\int_0^{\frac{\pi}2}|\sin^n(z)-\cos^n(z)|^{\frac{w}n}dz.
\]
Proceeding as before, we arrive at the integral
\[
\E(T_n^w)=\int_0^1 f_w(u)(1-u^n)^{\frac{w}{n}}du, \quad f_w(u)=\frac{4}{\pi}\cdot\frac{1}{(1+u^2)^{1+w/2}}.
\]
Again, $f_w$ can be written using the Appell seed of the generalized Euler
polynomials:
\[
\ap(t)=\left(  \frac{2}{1+e^{t}}\right)^{1+\frac{w}{2}},\quad
f_w(u)
=\frac{1}{\pi 2^{w/2-1}}\cdot\ap%
\big(2\ln(u)\big).
\]
This implies that $\beta_{\nu}$ is given by the generalized Euler
polynomials:
\[
\beta_{\nu}=\frac{2^{\nu+1-w/2}}{\pi}(-1)^{\nu}E_{\nu}^{(1+w/2)}(\frac{1}{2}).
\]
We summarize our findings below
\[%
\begin{split}
\E(T_n^w) &  \sim\int_{0}^{1}\frac{4}{\pi(1+u^{2})^{1+w/2}}\,du+\sum_{p=2}^{\infty
}\frac{a_{p}}{n^{p}}  =\frac2{\pi}B_{\frac12}(\frac12,\frac{w+1}2)\\
&\quad+\sum_{p=2}^{\infty}\frac{(-1)^{p-1}}{n^{p}}\sum_{\ell=1}^{p-1}%
\frac{2^{\ell-w/2}}{\pi}E_{\ell-1}^{(1+w/2)}(\frac{1}%
{2})\zeta(\ell+1,\{1\}_{p-\ell-1}),
\end{split}
\]
with $B_z(a,b)$ denote the incomplete Beta-function. We note that $T_n$ converges to a random variable $T_\infty$, with raw moments given by
\[
\E(T_n^w) = \int_0^{1}\frac{4}{\pi}\cdot\frac{1}{(1+u^2)^{1+w/2}}du =\frac{4}{\pi}\int_0^{\frac{\pi}4}\cos^w(t)dt
=\frac2{\pi}B_{\frac12}(\frac12,\frac{w+1}2).
\]
Finally, we mention that the distribution of $T_\infty=\max\{\sin(U,\cos(U))\}$, $U$ uniform on $[0,\pi/2]$, is given by an arcsin-law
\[
\mathbb{P}\{T_\infty \le x \}=\frac{4}\pi \arcsin(x)-1,\quad \frac{\sqrt{2}}2\le x\le 1.
\]
A intimately related variant is $S_{n}=\Vert\Big(\sin^n\big(\frac{\pi U}2  \big),\cos^n\big(\frac{\pi U}2 \big)\Big)\Vert_{n}$,
whose moments are 
\[
\E(S_n^w)=\frac{2}{\pi}\int_0^{\frac{\pi}2}\big(\sin^n(z)+\cos^n(z)\big)^{\frac{w}n}dz.
\]
Our results apply again all for non-zero real $w$.

\subsection{Norm of a random vector}
As noted in the introduction, see ~\eqref{def:norm}, the study of norms of
random vectors naturally gives rise to integrals of the form $I_{n}$.
We present two examples: the first yields asymptotic coefficients related to
generalized Euler polynomials, while the second leads to coefficients
associated with the probabilist's Hermite polynomials. Let $\left(  U,1-U\right)$ be random vector, where $U\sim$ Uniform$[0,1]$
distribution, and let $Z_{n}$ $=\Vert(U,1-U)\Vert_{n}$ denote its $n-$th
norm. Then
\[
\E(Z_{n})=\mathbb{E}\Big(U^{n}+(1-U)^{n}\Big)^{\frac{1}{n}}=\int
_{0}^{1}[x^{n}+(1-x)^{n}]^{\frac{1}{n}}dx.
\]
This integral also appeared as a problem in the American Mathematical
Monthly~\cite{AMM2016} and was treated by Louchard~\cite{Louchard}, as well
as~\cite{HKLL2020}. Using the symmetry around $x=\frac{1}{2}$, one can
simplify this to
\[
2\int_{0}^{\frac{1}{2}}[x^{n}+(1-x)^{n}]^{\frac{1}{n}}dx=2\int_{0}^{\frac
{1}{2}}(1-x)\left[  1+\left(  \frac{x}{1-x}\right)  ^{n}\right]  ^{\frac{1}%
{n}}dx.
\]
Now let $u=\frac{x}{1-x}$, or $x=\frac{u}{1+u}$. Then $dx=\frac{du}{(1+u)^{2}%
}$, and we have
\begin{align*}
\E(Z_{n})  &  =2\int_{0}^{1}\left(  1-\frac{u}{1+u}\right)
(1+u^{n})^{\frac{1}{n}}\frac{du}{(1+u)^{2}}\\
&  =\int_{0}^{1}f(u)(1+u^{n})^{1/n}\,du,\,\qquad\text{with }f(u)=\frac
{2}{(1+u)^{3}}.
\end{align*}
More generally, the $w-$th raw moment is
\[
\E(Z_{n}^{w})=\int_{0}^{1}f_{w}(u)(1+u^{n})^{\frac{w}{n}}%
du,\quad\text{with }f_{w}(u)=\frac{2}{(1+u)^{w+2}}.
\]
Our general theorem allows us to re-obtain the previous
results~\cite{HKLL2020}, with $\beta_{\nu}$ given in terms of the generalized
Euler polynomials, as we have
\[
\ap(t)=(\frac{2}{1+e^{t}})^{w+2},\quad f_{w}(u)=\frac{2}{(1+u)^{w+2}%
}=\frac{1}{2^{w+1}}\cdot\ap(\ln(u)),
\]
such that
\[
\beta_{\nu}=\frac{1}{2^{w+1}}(-1)^{\nu}E_{\nu}^{(w+2)}(1).
\]
or alternatively,
\begin{equation}
\beta_{\nu}=(-1)^{\nu}\sum_{j=0}^{\nu}\frac{(-1)^{j}j!\binom{w+1+j}{j}%
\genfrac{\{}{\}}{0pt}{}{\nu+1}{j+1}%
}{2^{w+1+j}}. \label{eqn:expansionZ_n2}%
\end{equation}
Consequently,
\begin{equation}
\E(Z_{n}^{w})~\sim a_{0}+\frac{1}{2^{w+1}}\sum_{p=2}^{\infty}%
\frac{(-1)^{p-1}}{n^{p}}\sum_{\ell=1}^{p-1}w^{p-\ell}E_{\ell-1}^{(w+2)}%
(1)\zeta(\overline{\ell+1},\{1\}_{p-\ell-1}), \label{eqn:expansionZ_n}%
\end{equation}
where
\[
a_{0}=\left\{
\begin{array}
[c]{c}%
\frac{2-2^{-w}}{1+w},\ \ \ \ \ \ if\ w\neq-1,\ \ \\
2\ln\left(  2\right)  ,\ \ \ \ \ if\ w=-1.
\end{array}
\right.
\]
For the expected value, case $w=1$, we can use our previous considerations
also to obtain an expansion in terms of ordinary zeta values, as we can use
Example~\ref{ex:GenEulerPoly}, generalized Euler polynomials with Appell seed
$\ap(t)=(\frac{2}{1+e^{t}})^{d}$, where $d=3$, such that
$f(u)=\frac{1}{4}\ap(\ln(u))$, satisfying the required assumption
$cd=3$. The results are in complete agreement with the previously obtained
numbers~\cite{HKLL2020}:
\[
\E(Z_{n})\sim\frac{3}{4}+\sum_{p=2}^{\infty}\frac{a_{p}}{n^{p}},
\]
with the first few concrete values given by
\begin{align*}
&  a_{2}=\frac{\pi^{2}}{48},\,a_{3}=\frac{\zeta(3)}{8},\,a_{4}=-\frac{\pi^{4}%
}{960},\,a_{5}=-\frac{\pi^{2}\zeta(3)}{48},\,a_{6}=\frac{83\pi^{6}}%
{241920}-\frac{\zeta(3)^{2}}{16}\\
&  a_{7}=\frac{3\pi^{4}\zeta(3)}{640}+\frac{\pi^{2}\zeta(5)}{32}+\frac
{3\zeta(7)}{16},\,a_{8}=-\frac{253\pi^{8}}{14515200}+\frac{5\pi^{2}%
\zeta(3)^{2}}{192}+\frac{3\zeta(3)\zeta(5)}{16}.
\end{align*}

We note in passing that from the theory of norms we anticipate the limit law
$Z_{\infty}$, with
\[
Z_{\infty}=\Vert(U,1-U)\Vert_{\infty}=\max\{U,1-U\}\sim\text{Uniform}[\frac
{1}{2},1].
\]
Our main result immediately leads to moment convergence plus the complete
asymptotic expansion in terms of alternating multiple zeta values, where
\[
\E(Z_{\infty}^{w})=\int_{0}^{1}\frac{2}{(1+u)^{w+2}}du=\frac{2\left(
1-\frac{1}{2^{w+1}}\right)  }{w+1}.
\]
Now to the second example.

Let $\left(  e^{\frac{Y}{2}},e^{-\frac{Y}{2}}\right)  $ be random vector with
$Y\sim\ \mathcal{N}\left(  0,1\right)  ,$ the standard normal distribution,
and let $\ Z_{n}$ $=\Vert\left(  e^{\frac{Y}{2}},e^{-\frac{Y}{2}}\right)
\Vert_{n}$ denote its $n-$th norm. The $w-$th raw moment of $Z_{n}$
is
\[
\E(Z_{n}^{w})=\int_{-\infty}^{\infty}\left(  e^{\frac{ny}{2}%
}+e^{-\frac{ny}{2}}\right)  ^{\frac{w}{n}}\frac{e^{-\frac{y^{2}}{2}}}%
{\sqrt{2\pi}}dy=\sqrt{\frac{2}{\pi}}\int_{0}^{\infty}e^{-\frac{y^{2}}{2}%
}\left(  e^{\frac{ny}{2}}+e^{-\frac{ny}{2}}\right)  ^{\frac{w}{n}}dy.
\]
Now we use the substitution
\[
y=-\ln\left(  u\right)  ,\ dy=-\frac{1}{u}du,\ \ e^{-y}=u.
\]
This gives,%
\[
\E(Z_{n}^{w})=\sqrt{\frac{2}{\pi}}\int_{0}^{1}e^{-\frac{\ln^{2}\left(
u\right)  }{2}}\left(  u^{\frac{n}{2}}+u^{-\frac{n}{2}}\right)  ^{\frac{w}{n}%
}\frac{du}{u}=\int_{0}^{1}f_{w}(u)(1+u^{n})^{w/n}\,du,\text{ }%
\]
with%
\[
f_{w}(u)=\sqrt{\frac{2}{\pi}}\frac{e^{-\frac{\ln\left(  u\right)  ^{2}}{2}}%
}{u^{\frac{w}{2}+1}}.
\]
For the second inverse moment value, corresponding to $w=-2$, we may apply
Example~\ref{ex:GenHermitePoly}.\newline In this case $\ap%
(t)=e^{-\frac{t^{2}}{2}}$, $f(u)=\sqrt{\frac{2}{\pi}}\ap(\ln(u))$,
satisfying the required symmetry condition for $w=-2.$

Therefore, the asymptotic expansion of the inverse square moment can be
written in terms of the probabilist's Hermite polynomials,
\begin{equation}
\E(Z_{n}^{-2})\sim\sqrt{e}\Erfc\left(  \frac{1}{\sqrt{2}%
}\right)  +\sum_{p=2}^{\infty}\frac{1}{n^{p}}\sum_{\nu=1}^{p-1}2^{p-\nu}%
\beta_{\nu-1}\zeta(\overline{\nu+1},\{1\}_{p-\nu-1}),
\end{equation}
where
\[
\beta_{\nu}=(-1)^{\nu}\sqrt{\frac{2}{\pi}}He_{\nu}(1).
\]
Since the solvability condition is met for $w=-2,$ then we can obtain an
expansion in terms of ordinary zeta values%
\[
\E(Z_{n}^{-2})\sim\sqrt{e}\Erfc\left(  \frac{1}{\sqrt{2}%
}\right)  +\sum_{p=2}^{\infty}\frac{1}{n^{p}}\sum_{\nu=1}^{p-1}\rho_{p,\nu
}\zeta(\nu+1,\{1\}_{p-\nu-1}),
\]
where%
\[
\rho_{p,\nu}=\left(  -1\right)  ^{\nu}\sum_{\ell=1}^{\nu}\binom{\nu-1}{\ell
-1}2^{\ell-1}\beta_{p-\ell-1}.
\]
The first few concrete values of the coefficients are given by%
\[
\E(Z_{n}^{-2})\sim \sqrt{e}\ \Erfc\left(  \frac{1}{\sqrt{2}}\right)+\sum_{p=2}^{\infty}\frac{a_{p}}{n^{p}},
\]%
\begin{align*}
a_{2}&=-\sqrt{\frac{2}{\pi}}\zeta\left(  2\right)  ,\quad
a_{3}=2\sqrt{\frac{2}{\pi}}\zeta(3),\quad a_{4}=-\sqrt{\frac{2}{\pi}}\frac
{1}{2}\zeta(4),\\
a_{5}  &  =\sqrt{\frac{2}{\pi}}\left(  -\frac{2\pi^{2}\zeta(3)}{3}%
+4\zeta(5)\right)  ,\quad a_{6}=\sqrt{\frac{2}{\pi}}\left(  -\frac{13}{8}%
\zeta(6)+4\zeta(3)^{2}\right)  .
\end{align*}

\subsection{Difference of random variables}

We study a counterpart of the random variable $Z_{n}$~\eqref{def:norm}. Let
\begin{equation}
Y_{n}=|U^{n}-(1-U)^{n}|^{\frac{1}{n}},\quad U=\text{Uniform}[0,1].
\label{def:rvDiff}%
\end{equation}
By the same arguments as before, we obtain for the expectation of $Y_{n}$ the
expression
\[
\E(Y_{n})=\E\Big|U^{n}-(1-U)^{n}\Big|^{\frac{1}{n}}=\int
_{0}^{1}|x^{n}-(1-x)^{n}|^{\frac{1}{n}}dx.
\]
Splitting again at $x=\frac{1}{2}$ and the previous substitutions give
\[
\E(Y_{n})=\int_{0}^{1}f(u)(1-u^{n})^{1/n}\,du,\,\qquad\text{with
}f(u)=\frac{2}{(1+u)^{3}}.
\]
Similarly, its $w-$th raw moment is given by
\[
\E(Y_{n}^{w})=\int_{0}^{1}f_{w}(u)(1-u^{n})^{\frac{w}{n}}du,\quad
f_{w}(u)=\frac{2}{(1+u)^{w+2}}.
\]
Our general result applies again and provides a detailed moment convergence,
almost identical to~\eqref{eqn:expansionZ_n}, but with non-alternating MZVs:
\begin{equation}
\E(Y_{n}^{w})~\sim\frac{2\left(  1-\frac{1}{2^{w+1}}\right)  }%
{w+1}+\sum_{p=2}^{\infty}\frac{1}{n^{p}}\sum_{\ell=1}^{p-1}w^{p-\ell}%
\beta_{\ell-1}(-1)^{p-\ell}\zeta(\ell+1,\{1\}_{p-\ell-1}),
\label{eqn:expansionY_n}%
\end{equation}
with $\beta_{\nu}$ as stated in~\eqref{eqn:expansionZ_n2}. For the interested
reader we note that
\[
|U^{n}-(1-U)^{n}|=\max(U,1-U)^{n}\left\vert 1-\left(  \frac{\min(U,1-U)}%
{\max(U,1-U)}\right)  ^{n}\right\vert .
\]
Consequently, taking the $n$-th root gives
\[
Y_{n}=\max(U,1-U)\left\vert 1-\left(  \frac{\min(U,1-U)}{\max(U,1-U)}\right)
^{n}\right\vert ^{1/n}.
\]
Thus, $Y_{n}$ converges
to $\max(U,1-U)$ and the moment convergence highlights the asymptotics.

\bibliographystyle{siam}
\bibliography{AIntAMZV-ref}

\begin{thebibliography}{10}

\bibitem{BBB}
{\sc J.~M. Borwein, D.~M. Bradley, and D.~J. Broadhurst}, {\em Evaluation of
  $k$-fold euler/zagier sums: a compendium of results for arbitrary $k$},
  Electron. J. Combin., 4 (1997).
\newblock res. art. 5.

\bibitem{Charlton}
{\sc S.~Charlton, H.~Gangl, and D.~Radchenko}, {\em On functional equations for
  nielsen polylogarithms}, Communications in Number Theory and Physics, 15
  (2021), pp.~363--454.

\bibitem{AMM2016}
{\sc O.~Furdui}, {\em Problem 11941}, Amer. Math. Monthly, 123 (2016).

\bibitem{AMM2025}
{\sc O.~Furdui and A.~S\^{i}nt\u{a}m\u{a}rian}, {\em Problem 12550}, Amer.
  Math. Monthly, 131 (2025).

\bibitem{H1992}
{\sc M.~E. Hoffman}, {\em Multiple harmonic series}, Pacific J. Math., 152
  (1992), pp.~275--290.

\bibitem{HKLL2020}
{\sc M.~E. Hoffman, M.~Kuba, M.~Levy, and G.~Louchard}, {\em An asymptotic
  series for an integral}, Ramanujan Journal 53, 53 (2020), pp.~1--25.

\bibitem{Knuth}
{\sc D.~E. Knuth}, {\em The Art of Computer Programming, Volume 3: Sorting and
  Searching}, Addison-Wesley, Reading, MA, 1973.

\bibitem{Kolbig}
{\sc K.~S. K\"olbig}, {\em Nielsen's generalized polylogarithms}, SIAM J. Math.
  Anal., 17 (1986), pp.~1232--1258.

\bibitem{Louchard}
{\sc G.~Louchard}, {\em Two applications of polylog functions and euler sums}.
\newblock ArXiv:1709.08686, 2017.

\bibitem{Nielsen}
{\sc N.~Nielsen}, {\em Der {E}ulersche {D}ilogarithmus und seine
  {V}erallgemeinerungen}, Nova Acta Leopoldina, 90 (1909), p.~123–211.

\bibitem{Watson}
{\sc G.~N. Watson}, {\em The harmonic functions associated with the parabolic
  cylinder}, Proceedings of the London Mathematical Society, 2 (1918),
  p.~116–148.

\bibitem{Z1992}
{\sc D.~Zagier}, {\em Values of zeta functions and their applications}, Progr.
  Math., First European Congress of Mathematics (Paris, 1992), vol. II (A.
  Joseph et al., eds.), 120 (1994), pp.~497--512.

\end{thebibliography}
{}

\end{document}